\documentclass[11pt,a4paper]{article}

\usepackage{amssymb,amsmath,amsthm}
\usepackage{braket,color}
\usepackage[all]{xy}

\theoremstyle{plain}
\newtheorem{thm}{Theorem}[section]
\newtheorem{prop}[thm]{Proposition}
\newtheorem{lem}[thm]{Lemma}
\newtheorem{cor}[thm]{Corollary}

\theoremstyle{definition}
\newtheorem{dfn}[thm]{Definition}
\newtheorem{rem}[thm]{Remark}

\numberwithin{equation}{section}

\makeatletter
\renewenvironment{proof}[1][\proofname]{\par
  \pushQED{\qed}%
  \normalfont \topsep6\p@\@plus6\p@\relax
  \trivlist
  \item[\hskip\labelsep
	\bfseries
    #1\@addpunct{.}]\ignorespaces
}{%
  \popQED\endtrivlist\@endpefalse
}
\makeatother

\DeclareMathOperator{\Hom}{Hom}
\DeclareMathOperator{\End}{End}
\DeclareMathOperator{\Aut}{Aut}
\DeclareMathOperator{\Ker}{Ker}
\DeclareMathOperator{\Ima}{Im}
\DeclareMathOperator{\Res}{Res}
\DeclareMathOperator{\ev}{ev}
\DeclareMathOperator{\id}{id}

\DeclareMathOperator{\ad}{ad}
\DeclareMathOperator{\ch}{ch}
\DeclareMathOperator{\Ann}{Ann}

\newcommand{\ve}{\varepsilon}

\newcommand{\bbZ}{\mathbb{Z}}
\newcommand{\bbC}{\mathbb{C}}

\newcommand{\frakS}{\mathfrak{S}}

\newcommand{\frg}{\mathfrak{g}}
\newcommand{\frh}{\mathfrak{h}}

\newcommand{\frt}{\mathfrak{t}}
\newcommand{\frs}{\mathfrak{s}}

\newcommand{\tor}{\mathfrak{g}_{\mathrm{tor}}}
\newcommand{\aff}{\mathfrak{g}_{\mathrm{aff}}}
\newcommand{\affh}{\mathfrak{h}_{\mathrm{aff}}}
\newcommand{\affn}{\mathfrak{n}_{\mathrm{aff}}}
\newcommand{\affnbar}{\bar{\mathfrak{n}}_{\mathrm{aff}}}
\newcommand{\torh}{\mathfrak{a}_{\mathrm{tor}}}
\newcommand{\torn}{\mathfrak{n}_{\mathrm{tor}}}
\newcommand{\tornbar}{\bar{\mathfrak{n}}_{\mathrm{tor}}}

\newcommand{\affI}{I_{\mathrm{aff}}}

\newcommand{\affQ}{Q_{\mathrm{aff}}}
\newcommand{\affDelta}{\Delta_{\mathrm{aff}}}

\newcommand{\affF}{\mathcal{F}_{\mathrm{aff}}}

\newcommand{\calH}{\mathcal{H}}

\newcommand{\glob}{W_{\mathrm{glob}}}
\newcommand{\loc}{W_{\mathrm{loc}}}

\newcommand{\bbV}{\mathbb{V}}
\newcommand{\bfv}{\mathbf{v}}

\title{Level one Weyl modules for toroidal Lie algebras}

\author{Ryosuke Kodera}

\date{}

\makeatletter
\let\@old@@maketitle=\@maketitle
\def\@maketitle{%
\footnotetext{%
\hspace*{-1em}\hspace*{-\footnotesep}%
E-mail address: kodera@math.s.chiba-u.ac.jp
}
\@old@@maketitle
}
\makeatother

\begin{document}
\maketitle

\begin{abstract}
We identify level one global Weyl modules for toroidal Lie algebras with certain twists of modules constructed by Moody-Eswara Rao-Yokonuma via vertex operators for type ADE and by Iohara-Saito-Wakimoto and Eswara Rao for general type.
The twist is given by an action of $\mathrm{SL}_{2}(\mathbb{Z})$ on the toroidal Lie algebra.
As a byproduct, we obtain a formula for the character of the level one local Weyl module over the toroidal Lie algebra and that for the graded character of the level one graded local Weyl module over an affine analog of the current Lie algebra.
\end{abstract}

\section{Introduction}

\subsection{Motivation}\label{subsection:1.1}

We study global/local Weyl modules for toroidal Lie algebras and an affine analog of current Lie algebras.
The notion of Weyl modules for affine Lie algebras has been introduced by Chari-Pressley in \cite{MR1850556} as a family of integrable highest weight modules with a universal property.
Later Chari-Loktev initiated in \cite{MR2271991} to study Weyl modules for current Lie algebras in a graded setting.
The graded characters of local Weyl modules for current Lie algebras have been studied by many authors.
Now they are known to coincide with Macdonald polynomials specialized at $t=0$, a.k.a.\ $q$-Whittaker functions (Chari-Loktev~\cite{MR2271991}, Fourier-Littelmann~\cite{MR2323538}, Naoi~\cite{MR2855081}, Sanderson~\cite{MR1771615}, Ion~\cite{MR1953294}, Lenart-Naito-Sagaki-Schilling-Shimozono~\cite{MR3674171}).

Toroidal Lie algebras are natural generalization of affine Lie algebras.
For a finite-dimensional simple Lie algebra $\frg$, the corresponding toroidal Lie algebra $\tor$ is defined as the universal central extension of the double loop Lie algebra $\frg \otimes \bbC[s^{\pm 1}, t^{\pm 1}]$ with the degree operators.
We can also consider a Lie algebra $\tor^+$ which is defined by replacing $\bbC[s^{\pm 1}, t^{\pm 1}]$ with $\bbC[s, t^{\pm 1}]$.
See Section~\ref{subsection:toroidal} for precise definitions.
We expect that the characters of Weyl modules for $\tor$ and $\tor^+$ produce a very interesting class of special functions.
In this article, we study the first nontrivial example: the Weyl module associated with the level one dominant integral weight.

A big difference between the toroidal and the affine Lie algebra is the structure of their centers.
The toroidal Lie algebra without the degree operators has an infinite-dimensional center, while the center of the affine Lie algebra is one-dimensional.
The Weyl modules are examples of modules over the toroidal Lie algebra on which the action of the center does not factor a finite-dimensional quotient.
We note that Chari-Le have studied in \cite{MR2017585} local Weyl modules for a quotient of the toroidal Lie algebra.
The resulting quotient is an extension of the double loop Lie algebra by a two-dimensional center with the degree operators.
In particular, the Weyl modules considered in this article are possibly bigger than those studied in \cite{MR2017585} (See \ref{subsection:1.3} below).

\subsection{Outline}\label{subsection:1.2}

Let us summarize contents and results of the article.
In Section~\ref{section:Preliminaries}, we introduce the main object: the toroidal Lie algebra $\tor$.
We also introduce an affine analog of the current Lie algebra which is denoted by $\tor^+$.
Then we recall their basic properties.
Among other things, a certain automorphism of $\tor$ will play an important role.
The ring $\bbC[s^{\pm 1}, t^{\pm 1}]$ admits an $\mathrm{SL}_2(\mathbb{Z})$-action by the coordinate change.
This action naturally induces automorphisms of $\tor$.
We denote by $S$ the automorphism corresponding to the $S$-transformation.

In Section~\ref{section:Weyl modules}, we define the global and the local Weyl modules following \cite{MR1850556}, \cite{MR2271991}, \cite{MR2102326}, \cite{MR2718936}, \cite{MR2017585}.
The global Weyl module $\glob(\Lambda)$ for $\tor$ is attached to each dominant integral weight $\Lambda$ of the affine Lie algebra. 
We identify the endomorphism ring of $\glob(\Lambda)$ with a symmetric Laurent polynomial ring $A(\Lambda)$ in Proposition~\ref{prop:endomorphism} and define the local Weyl module $\loc(\Lambda,\mathbf{a})$ for each maximal ideal $\mathbf{a}$ of $A(\Lambda)$.
The argument is similar to known one for the affine and the current Lie algebras.
The global/local Weyl modules $\glob^+(\Lambda)$ and $\loc^+(\Lambda,\mathbf{a})$ for $\tor^+$ are similarly defined.
We prove in Proposition~\ref{prop:weight} a finiteness property for weight spaces of the Weyl modules.
By this property, the characters of the local Weyl modules are well-defined.
This result has been established for the case of the affine Lie algebra in \cite{MR1850556} and for a quotient of the toroidal Lie algebra in \cite{MR2017585}. 
We remark that we need to investigate the action of the infinite-dimensional center, which is not treated in \cite{MR2017585}.
Then we turn to a special case where $\Lambda$ is of level one.
By the diagram automorphism, we can reduce the general level one case to that for the basic level one weight $\Lambda_0$.
Therefore we only consider the case of $\Lambda_0$ in the sequel.
We give an upper bound for the graded character of the level one local Weyl module $\loc^+(\Lambda_0,0)$ over $\tor^+$ in Proposition~\ref{prop:upper_bound}.

In Section~\ref{section:Vertex operator construction}, we prove an isomorphism between the level one global Weyl module $\glob(\Lambda_0)$ over the toroidal Lie algebra $\tor$ and the twist of a module $\bbV(0)$ by the automorphism $S^{-1}$, where $\bbV(0)$ has been constructed in works of Moody-Eswara Rao-Yokonuma~\cite{MR1066569}, Iohara-Saito-Wakimoto~\cite{MR1688100} and Eswara Rao \cite{MR3076215}.
This is our main theorem.

\begin{thm}[Theorem~\ref{thm:main}]
We have an isomorphism
\[
	\glob(\Lambda_0) \stackrel{\cong}{\longrightarrow} (S^{-1})^*\bbV(0)
\]
of $\tor$-modules.
\end{thm}

As a byproduct, we prove that the upper bound in Proposition~\ref{prop:upper_bound} indeed gives the characters of the level one local Weyl modules (see Section~\ref{subsection:Characters} for the definition of $\ch_p$ and $\ch_{p,q}$).
\begin{cor}[Corollary~\ref{cor:character}]
We have
\[
	\ch_{p} \loc(\Lambda_0,a) = \ch_{p} \loc^+(\Lambda_0,a) = \ch_p L(\Lambda_0) \left( \prod_{n>0} \dfrac{1}{1-p^n} \right)
\]
for $a \in \bbC^{\times}$ and
\[
	\ch_{p,q} \loc^+(\Lambda_0,0) = \ch_p L(\Lambda_0) \left( \prod_{n>0} \dfrac{1}{1-p^n q} \right).
\]
Here $L(\Lambda_0)$ is the level one integrable irreducible module of the affine Lie algebra with highest weight $\Lambda_0$. 
\end{cor}

\subsection{Related works}\label{subsection:1.3}

Let us give two comments regarding other works.
The first one is for \cite{MR2017585} mentioned earlier.
In \cite{MR2017585}, Chari-Le have studied local Weyl modules for some quotients of $\tor$ and $\tor^+$.
They have proved that the level one local Weyl modules in their setting are irreducible and are isomorphic to the evaluation modules \cite[Theorem~4]{MR2017585}.
Hence we see by our results that the level one local Weyl modules for $\tor$ and $\tor^+$ are bigger than those studied in \cite{MR2017585}.
We remark that one of our results (Proposition~\ref{prop:upper_bound}) gives an alternative proof of \cite[Theorem~4]{MR2017585}.

The second one is for \cite{MR3908899}.
In \cite[Theorem~3.8]{MR3908899}, Tsymbaliuk has proved that the level one Fock representation of Saito-Takemura-Uglov \cite{MR1603798} and Feigin-Jimbo-Miwa-Mukhin \cite{MR3023228} over the quantum toroidal algebra of type A is isomorphic to a twist of the vertex representation of Saito \cite{MR1617066}.
Here the twist is given by an automorphism analogous to $S^{-1}$ which has been constructed by Miki \cite{MR1693755}.
This result motivated the present work.
In the situation of \cite{MR3908899}, both the Fock and the vertex representations are known to be irreducible and hence it can be checked by comparing their highest weights to show the isomorphism.
Thus, although the calculation of $S^{-1}$ in the quantum toroidal case is much more involved, the argument to show the isomorphism is simple.
It is an interesting problem to establish results analogous to this article for quantum toroidal algebras and affine Yangians.

\subsection*{Acknowledgments}
The author is grateful to Ryo Sato who pointed out that the result of \cite{MR3076215} can be used to improve this work.
He also would like to thank Yoshihisa Saito and Kentaro Wada for helpful discussion. 
This work was supported by JSPS KAKENHI Grant Number 17H06127 and 18K13390.

\section{Preliminaries}\label{section:Preliminaries}

\subsection{Simple Lie algebras}

Let $\frg$ be a finite-dimensional simple Lie algebra over $\bbC$ with a fixed Cartan subalgebra $\frh$.
We also fix a Borel subalgebra containing $\frh$.
The index set of simple roots is denoted by $I$.
Let $\alpha_i$ ($i \in I$) be simple roots.
We denote by $\Delta$, $\Delta^+$, $\Delta^-$ the sets of roots, positive roots, negative roots, respectively.
Let $\frg_{\alpha}$ ($\alpha \in \Delta)$ be the corresponding root space and put $\frg_0 = \frh$.
The highest root is denoted by $\theta$.

Let $(\,,\,)$ be a nondegenerate invariant symmetric bilinear form on $\frg$.
We denote by the same letter the bilinear form on $\frh^*$ induced from $(\,,\,)$ and normalize them by $(\theta,\theta)=2$.
Put $d_i = (\alpha_i,\alpha_i)/2$.
We fix Chevalley generators $e_i, f_i, h_i$ ($i \in I$) so that $(e_i,f_i)=d_i^{-1}$ and $h_i = [e_i,f_i]$.
We also fix root vectors $e_{\theta} \in \frg_{\theta}$ and $f_{\theta} \in \frg_{-\theta}$ so that $(e_{\theta},f_{\theta})=1$.
We denote by $h_{\alpha} \in \frh$ the coroot corresponding to $\alpha \in \Delta$.
The root lattice $Q$ is defined by $Q=\bigoplus_{i \in I} \bbZ \alpha_i$.

\subsection{Toroidal Lie algebras}\label{subsection:toroidal}

The universal central extension of the Lie algebra $\frg \otimes \bbC[s^{\pm 1},t^{\pm 1}]$ is given by
\[
	\frg \otimes \bbC[s^{\pm 1},t^{\pm 1}] \oplus \Omega_{\bbC[s^{\pm 1},t^{\pm 1}]} / \Ima d.
\]
Here $\Omega_A$ for a commutative $\bbC$-algebra $A$ denotes the module of differentials, and $d \colon A \to \Omega_A$ the differential map.
The Lie bracket is given by
\[
	[x \otimes a, y \otimes b] = [x,y] \otimes ab + (x,y) (da)b.
\]
See \cite[Section~2]{MR1066569} for details.

We put
\[
	c(k,l) = \begin{cases}
		s^k t^{l-1} dt & \text{if } k \neq 0,\\
		s^{-1} t^l ds & \text{if } k = 0
	\end{cases}
\]
for $(k,l) \in \bbZ^2 \setminus \{(0,0)\}$ and $c_s = s^{-1} ds$, $c_t = t^{-1} dt$. 
Then $\Omega_{\bbC[s^{\pm 1},t^{\pm 1}]} / \Ima d$ has a $\bbC$-basis $c(k,l)$ with $(k,l) \in \bbZ^2 \setminus \{(0,0)\}$, $c_s$, $c_t$.
We can explicitly describe the Lie bracket as follows:
\begin{equation}
	\begin{split}
		&[x \otimes s^k t^l, y \otimes s^m t^n] \\
		&= \begin{cases}
			[x,y] \otimes s^{k+m} t^{l+n} + (x,y) \dfrac{lm-kn}{k+m} c(k+m,l+n) & \text{if } k+m \neq 0,\\
			[x,y] \otimes t^{l+n} + (x,y) k c(0,l+n) & \text{if } k+m = 0 \text{ and } l+n \neq 0,\\
			[x,y] \otimes 1 + (x,y) ( k c_s + l c_t ) & \text{if } k+m = 0 \text{ and } l+n = 0.
		\end{cases}\label{eq:bracket}
	\end{split}
\end{equation}
We add the degree operators $d_s$, $d_t$ to this central extension and define the toroidal Lie algebra $\tor$ by
\[
	\tor = \frg \otimes \bbC[s^{\pm 1},t^{\pm 1}] \oplus \bigoplus_{(k,l) \in \bbZ^2 \setminus \{(0,0)\}} \bbC c(k,l) \oplus \bbC c_s \oplus \bbC c_t \oplus \bbC d_s \oplus \bbC d_t,
\]
where the additional commutation relations are as follows:
\begin{gather*}
	[d_s, x \otimes s^k t^l] = k x \otimes s^k t^l, \quad [d_t, x \otimes s^k t^l] = l x \otimes s^k t^l, \\
	[d_s, c(k,l)] = k c(k,l), \quad [d_t, c(k,l)] = l c(k,l),\\
	[d_s,c_s]=[d_t,c_s]=[d_s,c_t]=[d_t,c_t]=[d_s,d_t]=0.
\end{gather*}

\begin{rem}
Note that we have
\[
	c(k,l) = \begin{cases}
		(-k/l) s^{k-1} t^{l} ds & \text{if } k \neq 0,\\
		s^{-1} t^l ds & \text{if } k = 0
	\end{cases}
\]
for $l \neq 0$.
In particular, $c(k+1,l)$ is a nonzero multiple of $s^{k} t^{l} ds$ if $l \neq 0$. 
We will use this fact throughout the article.
\end{rem}

Let $\tor'$ be the Lie subalgebra of $\tor$ without $d_s$:
\[
	\tor' = \frg \otimes \bbC[s^{\pm 1},t^{\pm 1}] \oplus \bigoplus_{(k,l) \in \bbZ^2 \setminus \{(0,0)\}} \bbC c(k,l) \oplus \bbC c_s \oplus \bbC c_t \oplus \bbC d_t.
\]
We also consider the following Lie subalgebra $\tor^+$ of $\tor$:
\[
	\tor^+ = \frg \otimes \bbC[s,t^{\pm 1}] \oplus \bigoplus_{\substack{k \geq 1\\l \in \bbZ}} \bbC c(k,l) \oplus \bbC c_t \oplus \bbC d_t.
\]
The Lie algebra $\tor^+$ is the semidirect product of the universal central extension of $\frg \otimes \bbC[s,t^{\pm 1}]$ and the one-dimensional abelian Lie algebra $\bbC d_t$.
It is an affine analog of the current Lie algebra $\frg \otimes \bbC[s]$ and has a $\bbZ_{\geq 0}$-graded Lie algebra structure by assigning
\[
	\deg (x \otimes s^k t^l) = k \ (x \in \frg),\quad \deg c(k,l) = k \ (k \geq 1, l \in \bbZ),\quad \deg c_t = \deg d_t = 0.
\]

\begin{rem}
Later we will study graded $\tor^+$-modules.
It is equivalent to considering modules of $\tor^+ \oplus \bbC d_s$. 
\end{rem}

The toroidal Lie algebra $\tor$ contains two Lie subalgebras $\aff^{(s)}$ and $\aff^{(t)}$ isomorphic to the affine Lie algebra associated with $\frg$:
\[
	\aff^{(s)} = \frg \otimes \bbC[s^{\pm 1}] \oplus \bbC c_s \oplus \bbC d_s, \quad \aff^{(t)} = \frg \otimes \bbC[t^{\pm 1}] \oplus \bbC c_t \oplus \bbC d_t.
\]
Note that $\tor^+$ contains $\aff^{(t)}$.
We have
\[
	\tor = \left(\aff^{(t)}\right)' \otimes \bbC[s^{\pm 1}] \oplus \bigoplus_{\substack{k \in \bbZ\\l \neq 0}} \bbC c(k,l) \oplus \bbC c_s \oplus \bbC d_s \oplus \bbC d_t,
\]
\[
	\tor^+ = \left(\aff^{(t)}\right)' \otimes \bbC[s] \oplus \bigoplus_{\substack{k \geq 1\\l \neq 0}} \bbC c(k,l) \oplus \bbC d_t,
\]
where $\left(\aff^{(t)}\right)' = \frg \otimes \bbC[t^{\pm 1}] \oplus \bbC c_t$.
Here, the elements $c(k,0)=s^k t^{-1} dt$ are regarded as $c_t \otimes s^k \in \left(\aff^{(t)}\right)' \otimes s^k$.

\begin{rem}\label{rem:CL}
Chari-Le~\cite{MR2017585} have studied a version of toroidal Lie algebras which is the quotient of $\tor$ modulo the elements $c(k,l)$ with $l \neq 0$, namely, it is equal to
\[
	\frg \otimes \bbC[s^{\pm 1},t^{\pm 1}] \oplus \bigoplus_{k \neq 0} \bbC c(k,0) \oplus \bbC c_s \oplus \bbC c_t \oplus \bbC d_s \oplus \bbC d_t
	=\left(\aff^{(t)}\right)' \otimes \bbC[s^{\pm 1}] \oplus \bbC c_s \oplus \bbC d_s \oplus \bbC d_t
\]
as a $\bbC$-vector space.
\end{rem}

We introduce presentations of $\tor$ and $\tor^+$.
Put $\affI = I \sqcup \{0\}$.
Let $(a_{ij})_{i,j \in \affI}$ be the Cartan matrix of $\aff^{(t)}$ and set $d_0 = 1$.
\begin{dfn}
Let $\frt$ be the Lie algebra generated by $e_{i,k}$, $f_{i,k}$, $h_{i,k}$ ($i \in \affI$, $k \in \bbZ$), $c_s$, $d_s$, $d_t$ subject to the following defining relations:
\begin{gather*}
	c_s :\text{central}, \quad [h_{i,k},h_{j,l}]=d_j^{-1} a_{ij} k \delta_{k+l,0} c_s, \quad [e_{i,k},f_{j,l}]=\delta_{ij} \left( h_{i,k+l} + d_i^{-1} k \delta_{k+l,0} c_s \right),\\
	[h_{i,k},e_{j,l}] = a_{ij} e_{j,k+l}, \quad [h_{i,k},f_{j,l}] = -a_{ij} f_{j,k+l},\\
	[e_{i,k},e_{i,l}] = 0, \quad [f_{i,k},f_{i,l}] = 0,\\
	(\ad e_{i,0})^{1-a_{ij}} e_{j,k} = 0, \quad (\ad f_{i,0})^{1-a_{ij}} f_{j,k} = 0, \quad (i \neq j)\\
	[d_s, e_{i,k}] = k e_{i,k}, \quad [d_s, f_{i,k}] = k f_{i,k}, \quad [d_s, h_{i,k}] = k h_{i,k},\\
	[d_t, e_{i,k}] = \delta_{i,0} e_{i,k}, \quad [d_t, f_{i,k}] = -\delta_{i,0} f_{i,k}, \quad [d_t, h_{i,k}] = 0,\\
	[d_s,d_t]=0.
\end{gather*}
\end{dfn}

\begin{dfn}
Let $\frs$ be the Lie algebra generated by $e_{i,k}$, $f_{i,k}$, $h_{i,k}$ ($i \in \affI$, $k \in \bbZ_{\geq 0}$), $d_t$ subject to the following defining relations:
\begin{gather*}
	[h_{i,k},h_{j,l}]=0, \quad [e_{i,k},f_{j,l}]=\delta_{ij} h_{i,k+l},\\
	[h_{i,k},e_{j,l}] = a_{ij} e_{j,k+l}, \quad [h_{i,k},f_{j,l}] = -a_{ij} f_{j,k+l},\\
	[e_{i,k},e_{i,l}] = 0, \quad [f_{i,k},f_{i,l}] = 0,\\
	(\ad e_{i,0})^{1-a_{ij}} e_{j,k} = 0, \quad (\ad f_{i,0})^{1-a_{ij}} f_{j,k} = 0, \quad (i \neq j)\\
	[d_t, e_{i,k}] = \delta_{i,0} e_{i,k}, \quad [d_t, f_{i,k}] = -\delta_{i,0} f_{i,k}, \quad [d_t, h_{i,k}] = 0.
\end{gather*}
\end{dfn}

\begin{thm}[\cite{MR1066569} Proposition~3.5, \cite{MR4014633} Proposition~4.4]
We have an isomorphism of Lie algebras $\frt \to \tor$ such that
\begin{gather*}
	e_{i,k} \mapsto \begin{cases}
		e_i \otimes s^k & \text{if } i \in I, \\
		f_{\theta} \otimes s^k t & \text{if } i =0,
	\end{cases}\quad 
	f_{i,k} \mapsto \begin{cases}
		f_i \otimes s^k & \text{if } i \in I, \\
		e_{\theta} \otimes s^k t^{-1} & \text{if } i =0,
	\end{cases}\\
	h_{i,k} \mapsto \begin{cases}
		h_i \otimes s^k & \text{if } i \in I, \\
		-h_{\theta} \otimes s^k + s^k t^{-1} dt & \text{if } i =0,
	\end{cases}\quad c_s \mapsto c_s,\quad d_s \mapsto d_s,\quad d_t \mapsto d_t.
\end{gather*}
Moreover this restricts to an isomorphism $\frs \to \tor^+$.
\end{thm}

Under the isomorphism, the elements $e_{i,0}, f_{i,0}, h_{i,0}$ are in the Lie subalgebra $\aff^{(t)}$ and identified with its Chevalley generators.
We sometimes denote them by $e_{i}, f_{i}, h_{i}$.
Note that $e_{i,k}$, $f_{i,k}$, $h_{i,k}$ ($i \in I$, $k \in \bbZ$), $c_s$, $d_s$ generate the Lie subalgebra $\aff^{(s)}$ of $\frt \cong \tor$.

We introduce notions for the affine Lie algebra $\aff^{(t)}$.
Let $\affn^{(t)}$ be the Lie subalgebra of $\aff^{(t)}$ generated by $e_i$ ($i \in \affI$), and $\affnbar^{(t)}$ that generated by $f_i$ ($i \in \affI$).
Set
\[
	\affh^{(t)} = \frh \oplus \bbC c_t \oplus \bbC d_t.
\]
The generator of imaginary roots is denoted by $\delta$.
We put $\alpha_0 = -\theta + \delta$ so that $\alpha_i$ ($i \in \affI$) forms simple roots of $\aff^{(t)}$.
We denote by $\affDelta$, $\affDelta^+$ the sets of roots, positive roots, respectively.
Let $\left(\aff^{(t)}\right)_{\alpha}$ ($\alpha \in \affDelta)$ be the corresponding root space.
The coroot is defined by $h_{\beta+l\delta}=h_{\beta}+lc_t$ for $\beta \in \Delta \cup \{0\}$ and $l \in \bbZ$.
We set $\affQ = \bigoplus_{i \in \affI} \bbZ \alpha_i$ and $\affQ^+ = \sum_{i \in \affI} \bbZ_{\geq 0} \alpha_i$. 

We say that an element $\Lambda$ of $\Hom_{\bbC} (\affh^{(t)},\bbC)$ is a dominant integral weight of $\aff^{(t)}$ if $\langle h_i, \Lambda\rangle \in \bbZ_{\geq 0}$ holds for any $i \in \affI$.
In this article, they are further assumed to satisfy $\langle d_t, \Lambda\rangle =0$ for simplicity.
Define the fundamental weights $\Lambda_i$ ($i \in \affI$) by $\langle h_j , \Lambda_i \rangle = \delta_{ij}$ and $\langle d_t, \Lambda_i \rangle = 0$.
We denote by $L(\Lambda)$ the irreducible $\aff^{(t)}$-module with highest weight $\Lambda$.
We will use the symbol $L(\Lambda)^{(s)}$ for the irreducible $\aff^{(s)}$-module with highest weight $\Lambda$.

\subsection{Triangular decomposition}

Let $\torn$ be the Lie subalgebra of $\tor$ generated by $e_{i,k}$ ($i \in \affI$, $k \in \bbZ$), and $\tornbar$ that generated by $f_{i,k}$ ($i \in \affI$, $k \in \bbZ$).
Set
\begin{equation*}
	\begin{split}
		\torh &= \frh \otimes \bbC[s^{\pm 1}] \oplus \displaystyle\bigoplus_{k \neq 0} \bbC c(k,0) \oplus \bbC c_s \oplus \bbC c_t \oplus \bbC d_s \oplus \bbC d_t \\
		&= \left(\frh \oplus \bbC c_t\right) \otimes \bbC[s^{\pm 1}] \oplus \bbC c_s \oplus \bbC d_s \oplus \bbC d_t.
	\end{split}
\end{equation*}

\begin{prop}
We have
\[
	\torn = \affn^{(t)} \otimes \bbC[s^{\pm 1}] \oplus \displaystyle\bigoplus_{\substack{k \in \bbZ \\ l \geq 1}} \bbC c(k,l),\quad
	\tornbar = \affnbar^{(t)} \otimes \bbC[s^{\pm 1}] \oplus \displaystyle\bigoplus_{\substack{k \in \bbZ \\ l \leq -1}} \bbC c(k,l).
\]
\end{prop}

\begin{proof}
Denote by $\torn'$ and $\tornbar'$ the right-hand sides.
Then we see by the formula of the Lie bracket (\ref{eq:bracket}) that $\torn \supset \torn'$ and $\tornbar \supset \tornbar'$.
We also see that $\tornbar + \torh + \torn = \tornbar \oplus \torh \oplus \torn$.
Since we have $\tor = \tornbar' \oplus \torh \oplus \torn'$, the assertion holds.
\end{proof}

In this article, we call
\[
	\tor = \tornbar \oplus \torh \oplus \torn
\]
the triangular decomposition of $\tor$.

In $\tor^+$, the elements $e_{i,k}$ ($i \in \affI$, $k \in \bbZ_{\geq 0}$) generate 
\[
	\torn \cap \tor^+ = \affn^{(t)} \otimes \bbC[s] \oplus \displaystyle\bigoplus_{\substack{k \geq 1 \\ l \geq 1}} \bbC c(k,l),
\]
and $f_{i,k}$ ($i \in \affI$, $k \in \bbZ_{\geq 0}$) generate 
\[
	\tornbar \cap \tor^+ = \affnbar^{(t)} \otimes \bbC[s] \oplus \displaystyle\bigoplus_{\substack{k \geq 1 \\ l \leq -1}} \bbC c(k,l).
\]
Further set
\[
	\torh' = \torh \cap \tor' = \left(\frh \oplus \bbC c_t\right) \otimes \bbC[s^{\pm 1}] \oplus \bbC c_s \oplus \bbC d_t.
\]

\subsection{Automorphisms}\label{subsection:auto}

Let $S$ be the ring automorphism of $\bbC[s^{\pm 1},t^{\pm 1}]$ defined by $s \mapsto t$, $t \mapsto s^{-1}$.
It naturally induces a Lie algebra automorphism of $\tor$ which is denoted by the same letter $S$.
Later we will rather use its inverse $S^{-1}$.
It corresponds to the assignment $s \mapsto t^{-1}$, $t \mapsto s$.
In particular we have
\[
	S^{-1}(c(k,l)) = \begin{cases}
		(k/l) c(l,-k) & \text{ if } k,l \neq 0,\\
		-c(l,0) & \text{ if } k=0,\\
		c(0,-k) & \text{ if } l=0,
	\end{cases}
	\quad S^{-1}(c_s) = -c_t, \quad S^{-1}(c_t) = c_s.
\]

We introduce Lie algebra automorphisms $T_0$ and $T_{\theta}$ of $\tor$ by
\[
	T_0 = \exp\ad e_0 \circ \exp\ad (-f_0) \circ \exp\ad e_0,
\]
\[
	T_{\theta} = \exp\ad e_{\theta} \circ \exp\ad (-f_{\theta}) \circ \exp\ad e_{\theta}.
\]
We can regard them as automorphisms of $\tor^+$ by restriction.

\begin{lem}\label{lem:induction}
We have $e_{\theta} \otimes s^k t^l = T_0 T_{\theta} (e_{\theta} \otimes s^k t^{l+2})$.
\end{lem}

\begin{proof}
By a direct calculation.
We use the following:
\begin{align*}
	T_{\theta} (e_{\theta} \otimes s^k t^{l+2}) &= - f_{\theta} \otimes s^k t^{l+2},\\
	\exp\ad e_0 (f_{\theta} \otimes s^k t^{l+2}) &= f_{\theta} \otimes s^k t^{l+2},\\
	\exp\ad (-f_0) (f_{\theta} \otimes s^k t^{l+2}) &= f_{\theta} \otimes s^k t^{l+2} - (h_{\theta} \otimes s^k t^{l+1}-s^kt^ldt) - e_{\theta} \otimes s^k t^{l},\\
	\exp\ad e_0 (h_{\theta} \otimes s^k t^{l+1}) &= h_{\theta} \otimes s^k t^{l+1} + 2 f_{\theta} \otimes s^k t^{l+2},\\
	\exp\ad e_0 (e_{\theta} \otimes s^k t^{l}) &= e_{\theta} \otimes s^k t^{l} - h_{\theta} \otimes s^k t^{l+1} + s^k t^l dt - f_{\theta} \otimes s^k t^{l+2}.
\end{align*}
\end{proof}

Let $M$ be a module of $\mathcal{A}=\tor,$ $\tor',$ or $\tor^+$ and assume that $M$ is integrable as a $\aff^{(t)}$-module.
Then $T_0, T_{\theta} \in \Aut M$ are similarly defined.
Moreover they satisfy
\[
	T_0(xv) = T_0(x)T_0(v), \quad T_{\theta}(xv) = T_{\theta}(x)T_{\theta}(v)
\]
for $x \in \mathcal{A}$ and $v \in M$.

The Lie algebra automorphism $\tau_a$ ($a \in \bbC$) of $\tor^+$ is induced from the map $s \mapsto s+a$.

\subsection{Characters}\label{subsection:Characters}

Let $M$ be a module of $\mathcal{A}=\tor,$ $\tor',$ or $\tor^+$ and regard it as a $\aff^{(t)}$-module by restriction.
For $\lambda \in \frh^*$ and $m \in \bbC$, let $M_{\lambda-m\delta}$ be the corresponding weight space.
In this article, we always assume that any $\aff^{(t)}$-module $M$ has the weight space decomposition and $M_{\lambda-m\delta}=0$ unless $m \in \bbZ$.

We define the $p$-character $\ch_p M$ of $M$ by
\[
	\ch_p M = \sum_{\substack{\lambda \in \frh^*\\ m \in \bbZ}} (\dim M_{\lambda-m\delta}) e^{\lambda} p^{m}
\]
if it is well-defined.
This is nothing but the ordinary $\aff^{(t)}$-character with $p=e^{-\delta}$.  
Let $M$ be a graded $\tor^+$-module and $M_{\lambda-m\delta} = \bigoplus_{n \in \bbZ} M_{\lambda-m\delta}[n]$ the decomposition of the weight space into graded pieces.
We define the $(p,q)$-character $\ch_{p,q} M$ of $M$ by
\[
	\ch_{p,q} M = \sum_{\substack{\lambda \in \frh^*\\ m,n \in \bbZ}} (\dim M_{\lambda-m\delta}[n]) e^{\lambda} p^{m} q^{n}
\] 
if it is well-defined.
For two formal sums 
\[
	f = \sum_{\substack{\lambda \in \frh^*\\ m \in \bbZ}} f_{\lambda,m} e^{\lambda} p^{m}, \quad g = \sum_{\substack{\lambda \in \frh^*\\ m \in \bbZ}} g_{\lambda,m} e^{\lambda} p^{m} \quad (f_{\lambda,m}, g_{\lambda,m} \in \bbZ),
\] 
we say $f \leq g$ if $f_{\lambda,m} \leq g_{\lambda,m}$ holds for all $\lambda$ and $m$.
We define an inequality $\leq$ for 
\[
	f = \sum_{\substack{\lambda \in \frh^*\\ m,n \in \bbZ}} f_{\lambda,m,n} e^{\lambda} p^{m}q^{n}, \quad g = \sum_{\substack{\lambda \in \frh^*\\ m,n \in \bbZ}} g_{\lambda,m,n} e^{\lambda} p^{m}q^{n} \quad (f_{\lambda,m,n}, g_{\lambda,m,n} \in \bbZ)
\] 
similarly.

\section{Weyl modules}\label{section:Weyl modules}

\subsection{Definitions of global/local Weyl modules}

\begin{dfn}
Let $\Lambda$ be a dominant integral weight of $\aff^{(t)}$.
The global Weyl module $\glob(\Lambda)$  for $\tor$ with highest weight $\Lambda$ is the $\tor$-module generated by $v_{\Lambda}$ subject to the following defining relations:
\begin{gather*}
	e_{i,k} v_{\Lambda} = 0\ (i \in \affI, k \in \bbZ),\quad h v_{\Lambda} = \langle h,\Lambda \rangle v_{\Lambda}\ (h \in \affh^{(t)}),\quad	f_i^{\langle h_i,\Lambda \rangle + 1} v_{\Lambda} = 0\ (i \in \affI), \label{eq:global1} \\
		c_s v_{\Lambda} = d_s v_{\Lambda} = 0. \label{eq:global2}
\end{gather*}
The global Weyl module $\glob^+(\Lambda)$ for $\tor^+$ with highest weight $\Lambda$ is the $\tor^+$-module generated by $v_{\Lambda}^+$ subject to the following defining relations:
\[
	e_{i,k} v_{\Lambda}^+ = 0\ (i \in \affI, k \in \bbZ_{\geq 0}),\quad h v_{\Lambda}^+ = \langle h,\Lambda \rangle v_{\Lambda}^+\ (h \in \affh^{(t)}),\quad 	f_i^{\langle h_i,\Lambda \rangle + 1} v_{\Lambda}^+ = 0\ (i \in \affI).
\]
\end{dfn}

We describe the endomorphism rings of $\glob(\Lambda)$ and $\glob^{+}(\Lambda)$.
The following argument is the same as in the case of the affine and the current Lie algebras.
We give some details for completeness.

\begin{lem}
We have an action $\varphi$ of $U(\torh')$ on each weight space $\glob(\Lambda)_{\Lambda-\beta}$ $(\beta \in \affQ^{+})$ defined by
\[
	\varphi(a) (X v_{\Lambda} ) = X (a v_{\Lambda})
\]
for $a \in U(\torh')$ and $X \in U(\tor')$.
\end{lem}

\begin{proof}
To see that the action is well-defined, we need to check that $X v_{\Lambda}=0$ implies $X (a v_{\Lambda})=0$.
By the same argument as \cite[3.4]{MR2718936}, we can show that if $v$ satisfies the relations 
\[
	e_{i,k} v = 0\ (i \in \affI, k \in \bbZ),\ h v = \langle h,\Lambda \rangle v\ (h \in \affh^{(t)}),\ f_i^{\langle h_i,\Lambda \rangle + 1} v = 0\ (i \in \affI),\ c_s v = 0,
\]
then so does $a v$.
This completes the proof.
\end{proof}

Let $\Ann v_{\Lambda}$ be the annihilator ideal of $U(\torh')$ and set
\[
	\tilde{A}(\Lambda) = U(\torh') / \Ann v_{\Lambda}.
\]
Since the action $\varphi$ of $\torh'$ factors through an abelian Lie algebra $\torh' / \bbC c_s \oplus \bbC d_t$, $\tilde{A}(\Lambda)$ is a commutative algebra.

\begin{lem}\label{lem:highest_weight_space}
The action map
\[
	\tilde{A}(\Lambda) \to \glob(\Lambda)_{\Lambda}, \quad a \mapsto a v_{\Lambda}
\]
gives an isomorphism of $\bbC$-vector spaces.
\end{lem}

\begin{proof}
The well-definedness and the injectivity immediately follow from the definition of $\tilde{A}(\Lambda)$.
The surjectivity holds since we have $\glob(\Lambda)_{\Lambda} = U(\torh') v_{\Lambda}$.
\end{proof}

\begin{lem}
The natural map
\[
	\tilde{A}(\Lambda) \to \End_{\tor'} \glob(\Lambda), \quad a \mapsto \varphi(a)
\]
gives an isomorphism of $\bbC$-algebras.
\end{lem}

\begin{proof}
By the definition of $\tilde{A}(\Lambda)$, we have a natural injective algebra homomorphism
\[
	\tilde{A}(\Lambda) \to \End_{\tor'} \glob(\Lambda), \quad a \mapsto \varphi(a).
\]
We also have a natural $\bbC$-linear map
\[
	\End_{\tor'} \glob(\Lambda) \to \glob(\Lambda)_{\Lambda}, \quad f \mapsto f(v_{\Lambda})
\]
and this is injective since $\glob(\Lambda)$ is generated by $v_{\Lambda}$.
The composite of the maps
\[
	\tilde{A}(\Lambda) \hookrightarrow \End_{\tor'} \glob(\Lambda) \hookrightarrow \glob(\Lambda)_{\Lambda}
\]
is given by $a \mapsto a v_{\Lambda}$.
Since this map is bijective by Lemma~\ref{lem:highest_weight_space}, the two injective maps are bijective.
\end{proof}

Write $\Lambda = \sum_{i \in \affI} m_i \Lambda_i$ with the fundamental weights $\Lambda_i$ and $m_i \in \bbZ_{\geq 0}$.
We define $A(\Lambda)$ by
\[
	A(\Lambda) = \bigotimes_{i \in \affI} \bbC[z_{i,1}^{\pm 1}, \ldots, z_{i,m_i}^{\pm 1}]^{\frakS_{m_i}},
\]	
the symmetric Laurent polynomial algebra associated with $\Lambda$.

\begin{prop}
The assignment
\[
	\sum_{m=1}^{m_i} z_{i,m}^k \mapsto h_{i,k}
\]
gives an isomorphism $A(\Lambda) \cong \tilde{A}(\Lambda)$ of $\bbC$-algebras.
\end{prop}

\begin{proof}
The well-definedness and the surjectivity of the map is proved in the same way as \cite[Proposition~1.1 (i), (iv), (v)]{MR1850556}.

We follow the argument in \cite[5.6]{MR3384485} to show the injectivity.
Take a nonzero element $a$ of $A(\Lambda)$ and fix a maximal ideal $\mathfrak{m}$ which does not contain $a$.
Assume that $\glob(\Lambda) \otimes_{A(\Lambda)} A(\Lambda) / \mathfrak{m}$ is nonzero.
Then the image of $a$ in $A(\Lambda) / \mathfrak{m}$ acts on $\glob(\Lambda) \otimes_{A(\Lambda)} A(\Lambda) / \mathfrak{m}$ by a nonzero scaler.
Hence we conclude that $a$ acts on $\glob(\Lambda)$ nontrivially and the map $A(\Lambda) \to \tilde{A}(\Lambda) \cong \End_{\tor'}\glob(\Lambda)$ is shown to be injective.

Thus it is enough to show that $\glob(\Lambda) \otimes_{A(\Lambda)} A(\Lambda) / \mathfrak{m}$ is nonzero.
We denote by $\bar{p}_{k}^{(i)}$ ($i \in \affI$, $k \in \bbZ$) the image of the power some function $p_{k}^{(i)} = \sum_{m=1}^{m_i} z_{i,m}^k$ in $A(\Lambda)/\mathfrak{m}$.
We can choose a set of nonzero complex numbers $\{ a_{i,m} \}$ satisfying
\[
	\sum_{m=1}^{m_i} a_{i,m}^k = \bar{p}_{k}^{(i)}
\]
under an identification $A(\Lambda)/\mathfrak{m} \cong \bbC$.
For each $a \in \bbC^{\times}$, we have the evaluation map
\[
	\ev_a \colon \tor' \to \aff^{(t)} 
\]
defined as the composite of
\[
	\tor' \to \tor' / \bigoplus_{\substack{k \in \bbZ\\ l \neq 0}} \bbC c(k,l) \oplus \bbC c_s \cong \left( \aff^{(t)} \right)' \otimes \bbC[s^{\pm 1}] \oplus \bbC d_t
\]
and the evaluation at $s=a$.
Then we have a nonzero $\tor'$-module homomorphism
\[
	\glob(\Lambda) \otimes_{A(\Lambda)} A(\Lambda) / \mathfrak{m} \to \bigotimes_{i \in \affI} \bigotimes_{m=1}^{m_i} \ev_{a_{i,m}}^{*}L(\Lambda_i)
\]
assigning $v_{\Lambda} \otimes 1$ to the tensor product of highest weight vectors.
This proves the assertion.
\end{proof}

We have a completely analogous story for the global Weyl module $\glob^+(\Lambda)$ over $\tor^+$ if we replace $A(\Lambda)$ with
\[
	A^+(\Lambda) = \bigotimes_{i \in \affI} \bbC[z_{i,1}, \ldots, z_{i,m_i}]^{\frakS_{m_i}}.
\]
We can summarize the discussion so far as follows.

\begin{prop}\label{prop:endomorphism}
We have $\End_{\tor'} \glob(\Lambda) \cong A(\Lambda)$ and $\End_{\tor^+} \glob^+(\Lambda) \cong A^+(\Lambda)$.
\end{prop}

For a maximal ideal $\mathbf{a}$ of $A = A(\Lambda)$ or $A^+(\Lambda)$, we denote by $\bbC_{\mathbf{a}}$ the corresponding one-dimensional module $A/\mathbf{a}$.

\begin{dfn}
We call
\[
	\loc(\Lambda,\mathbf{a}) = \glob(\Lambda) \otimes_{A(\Lambda)} \bbC_{\mathbf{a}}, \quad \loc^+(\Lambda,\mathbf{a}) = \glob^+(\Lambda) \otimes_{A^+(\Lambda)} \bbC_{\mathbf{a}}
\]
the local Weyl modules for $\tor'$ and $\tor^+$, respectively.
\end{dfn}
We denote the images of $v_{\Lambda}$ and $v_{\Lambda}^+$ in the local Weyl modules by $v_{\Lambda,\mathbf{a}}$ and $v_{\Lambda,\mathbf{a}}^+$.

\begin{rem}
The global/local Weyl modules for $\tor$ and $\tor^+$ can be regarded as a sort of highest weight modules with respect to their triangular decompositions:
\[
	\tor = \tornbar \oplus \torh \oplus \torn, \quad \tor^+ = \left( \tornbar \cap \tor^+ \right) \oplus \left( \torh \cap \tor^+ \right) \oplus \left( \torn\cap \tor^+ \right).
\]
\end{rem}

\subsection{Finiteness of weight spaces}\label{subsection:finiteness_property}

The goal of this subsection is to prove the following.

\begin{prop}\label{prop:weight}
\begin{enumerate}
\item
Every weight space $\glob(\Lambda)_{\Lambda-\beta}$ is finitely generated over $A(\Lambda)$.
Every weight space $\loc(\Lambda,\mathbf{a})_{\Lambda-\beta}$ is finite-dimensional.
\item
Every weight space $\glob^+(\Lambda)_{\Lambda-\beta}$ is finitely generated over $A^+(\Lambda)$.
Every weight space $\loc^+(\Lambda,\mathbf{a})_{\Lambda-\beta}$ is finite-dimensional.
\item
We have $\loc(\Lambda,\mathbf{a}) = U(\tor^+) v_{\Lambda,\mathbf{a}}$. 
\end{enumerate}
\end{prop}

We start to prove the following lemma.

\begin{lem}\label{lem:single}
Let $\Lambda$ be a dominant integral weight of $\aff^{(t)}$.
\begin{enumerate}
\item
For each positive root $\beta \in \affDelta^+$, there exists a nonnegative integer $N(\beta)$ satisfying the following:  we have
\[
	(X_{-\beta} \otimes s^{k}) v_{\Lambda} \in \sum_{m=0}^{N(\beta)} (X_{-\beta} \otimes s^{m}) A(\Lambda) v_{\Lambda}
\]
for any root vector $X_{-\beta}$ of $\affnbar^{(t)}$ corresponding to a negative root $-\beta$ and any $k$.

\item
For each positive integer $l >0$, there exists a nonnegative integer $N_l$ satisfying the following:  we have
\[
	c(k,-l) v_{\Lambda} \in \sum_{m=1}^{N_l} c(m,-l) A(\Lambda) v_{\Lambda} + \sum_{m=0}^{N_l} \left( \left( \aff^{(t)} \right)_{-l\delta} \otimes s^m\right) A(\Lambda) v_{\Lambda} 
\]
for any $k$.
\end{enumerate}
\end{lem}

\begin{proof}
The assertion (i) is proved in the same way as \cite[Proposition~3.2 and Corollary~3.1]{MR2017585}.  

We prove (ii).
Take an arbitrary element $\alpha$ of $\Delta^+$ and fix root vectors $x_{\alpha} \in \frg_{\alpha}$ and $x_{-\alpha} \in \frg_{-\alpha}$ satisfying $(x_{\alpha},x_{-\alpha})=1$.
Then we have
\begin{equation*}
	\begin{split}
		(s^{k} t^{-l} ds)v_{\Lambda} &= \left( [x_{\alpha} \otimes s, x_{-\alpha} \otimes s^{k}t^{-l}] - h_{\alpha} \otimes s^{k+1} t^{-l} \right) v_{\Lambda} \\
		&= (x_{\alpha} \otimes s) (x_{-\alpha} \otimes s^{k}t^{-l}) v_{\Lambda} - (h_{\alpha} \otimes s^{k+1} t^{-l}) v_{\Lambda}.
	\end{split}
\end{equation*} 
We have
\[
	(x_{\alpha} \otimes s) (x_{-\alpha} \otimes s^{k}t^{-l}) v_{\Lambda} \in (x_{\alpha} \otimes s) \sum_{m=0}^{N(\alpha + l\delta)} (x_{-\alpha} \otimes s^{m} t^{-l}) A(\Lambda) v_{\Lambda}
\]
by (i).
The right-hand side is equal to
\[
	\sum_{m=0}^{N(\alpha + l\delta)} (h_{\alpha} \otimes s^{m+1} t^{-l} + s^m t^{-l} ds ) A(\Lambda) v_{\Lambda} = \sum_{m=1}^{N(\alpha + l\delta)+1} (h_{\alpha} \otimes s^{m} t^{-l} + c(m,-l) ) A(\Lambda) v_{\Lambda}.
\]
We have
\[
	(h_{\alpha} \otimes s^{k+1} t^{-l}) v_{\Lambda} \in \sum_{m=0}^{N(l\delta)} (h_{\alpha} \otimes s^{m} t^{-l}) A(\Lambda) v_{\Lambda}
\]
again by (i).
Hence we conclude that
\[
	(s^{k} t^{-l} ds) v_{\Lambda} \in \sum_{m=1}^{N_l} c(m,-l) A(\Lambda) v_{\Lambda} + \sum_{m=0}^{N_l} \left( \left( \aff^{(t)} \right)_{-l\delta} \otimes s^m\right) A(\Lambda) v_{\Lambda}
\]
if we put $N_l = \max(N(l\delta),N(\alpha+l\delta)+1)$.
\end{proof}

The following proposition is an analog of \cite[Proposition~1.2]{MR1850556} for the case of the affine Lie algebra and of \cite[Proposition~3.2 and Corollary~3.1]{MR2017585} for the quotient of $\tor$ modulo the elements $c(k,l)$ with $l \neq 0$ (cf.\ Remark \ref{rem:CL}).

\begin{prop}\label{prop:span}
For each positive root $\beta_j \in \affDelta^+$ and each positive integer $l >0$, there exist nonnegative integers $N(\beta_j)$ and $N_l$ such that the weight space $\glob(\Lambda)_{\Lambda-\beta}$ for $\beta \in \affQ^+$ is spanned by elements of the form
\begin{equation}
	(X_{-\beta_1} \otimes s^{k_1}) \cdots (X_{-\beta_a} \otimes s^{k_a}) \left( \prod_{j=1}^{b} c(m_j,-l_j) \right) A(\Lambda) v_{\Lambda}, \label{eq:span}
\end{equation}
where each $X_{-\beta_{j}}$ is a root vector of $\affnbar^{(t)}$ corresponding to a negative root $-\beta_j$ and each $l_j > 0$ is a positive integer satisfying $\beta = \sum_{j=1}^a \beta_j + \left(\sum_{j=1}^b l_j \right) \delta$ and $0 \leq k_j \leq N(\beta_j)$, $1 \leq m_j \leq N_{l_j}$.
A similar statement also holds for $\glob^+(\Lambda)_{\Lambda-\beta}$. 
\end{prop}

\begin{proof}
By the PBW theorem, we see that $\glob(\Lambda)_{\Lambda-\beta}$ is spanned by elements of the form as (\ref{eq:span}) without any conditions on $k_j$ and $m_j$.
Then we use Lemma~\ref{lem:single} to show the assertion by the induction on $a+b$.  
\end{proof}

Thus we establish Proposition~\ref{prop:weight} from Proposition~\ref{prop:span}.
We also have the following.

\begin{prop}\label{prop:character}
Let $\mathbf{a}$ be a maximal ideal of $A(\Lambda)$ and regard it also as a maximal ideal of $A^{+}(\Lambda)$.
Then we have $\ch_p \loc^+(\Lambda,\mathbf{a}) \geq \ch_p \loc(\Lambda,\mathbf{a})$.
\end{prop}
\begin{proof}
We have a $\tor^+$-homomorphism $\loc^+(\Lambda,\mathbf{a}) \to \Res \loc(\Lambda,\mathbf{a})$ assigning $v_{\Lambda,\mathbf{a}}^+ \mapsto v_{\Lambda,\mathbf{a}}$.
It is surjective by Proposition~\ref{prop:weight} (iii).
\end{proof}

\subsection{Upper bound for the level one Weyl module}

In this subsection, we consider the case $\Lambda=\Lambda_0$.
The ring $A(\Lambda_0)$ is identified with $\bbC[z^{\pm 1}]$ and the action on $\glob(\Lambda_0)$ is given by 
\[
	z^k (X v_{\Lambda_0}) = X (h_{0,k} v_{\Lambda_0})
\]
for $X \in U(\tor')$.
This identification induces $A^+(\Lambda_0) = \bbC[z]$.

\begin{lem}\label{lem:h_{i,k}}
We have $h_{i,k} v_{\Lambda_0} = 0$ for $i \in I$ and $k \in \bbZ$.
\end{lem}

\begin{proof}
The defining relations $e_{i,k} v_{\Lambda_0}=0$ and $f_i v_{\Lambda_0} = 0$ for $i \in I$ imply the assertion. 
\end{proof}

Recall that $\sum_{i \in \affI} h_{i,k} = s^k t^{-1} dt$.
By Lemma~\ref{lem:h_{i,k}}, we see that the action of $A(\Lambda_0)$ on $\glob(\Lambda_0)$ is given by $z^k \mapsto s^k t^{-1} dt$.
In particular, $z$ acts by $c(1,0)=st^{-1}dt$.

We have defined the local Weyl modules $\loc(\Lambda_0,a)$ for $a \in \bbC^{\times}$ and $\loc^+(\Lambda_0,a)$ for $a \in \bbC$ by
\[
	\loc(\Lambda_0,a) = \glob(\Lambda_0) \otimes_{A(\Lambda_0)} \bbC_a, \quad \loc^+(\Lambda_0,a) = \glob^+(\Lambda_0) \otimes_{A^+(\Lambda_0)} \bbC_a.
\]
\begin{prop}\label{prop:independent}
The p-character $\ch_p \loc^+(\Lambda_0,a)$ is independent of $a \in \bbC$.
\end{prop}

\begin{proof}
The defining relations of $\loc^+(\Lambda_0,a)$ are given by
\begin{gather*}
	(\torn \cap \tor^+) v_{\Lambda_0,a}^+ = 0,\quad h_{i,k} v_{\Lambda_0,a}^+ = \delta_{i,0} a^k v_{\Lambda_0,a}^+ \ (i \in \affI, k \geq 0), \quad d_t v_{\Lambda_0,a}^+ = 0,\\
	f_0^2 v_{\Lambda_0,a}^+ = 0,\quad f_i v_{\Lambda_0,a}^+ = 0 \ (i \in I). 
\end{gather*}
Hence we have $\tau_a^*\loc^+(\Lambda_0,0) \cong \loc^+(\Lambda_0,a)$, where $\tau_a$ is the automorphism of $\tor^+$ defined in Section~\ref{subsection:auto}.
This proves the assertion.  
\end{proof}

We put 
\[
	W(\Lambda_0)=\loc^+(\Lambda_0,0) = \glob^+(\Lambda_0) \otimes_{A^+(\Lambda_0)} \bbC_0
\]
and denote its highest weight vector $v_{{\Lambda_0},0}^+$ by $v_0$.
This $W(\Lambda_0)$ is regarded as a graded $\tor^+$-module by setting $\deg v_0 = 0$. 

\begin{lem}\label{lem:f}
We have $f_{i,k} v_0 = 0$ for any $i \in \affI$ and $k \geq 1$.
\end{lem}

\begin{proof}
The assertion for $i \in I$ follows from $f_i v_0 =0$ and $h_{i,k} v_0 =0$.
The assertion for $i = 0$ follows from
\[
	0 = e_{0,k} f_0^2 v_0 = [e_{0,k}, f_0^2] v_0 = (-2f_{0,k} + 2 f_0 h_{0,k}) v_0 
\]
and $h_{0,k} v_0 =0$ for $k \geq 1$.
\end{proof}

\begin{lem}\label{lem:key}
Let $k \geq 1$.
We have
\begin{enumerate}
\item
\[
	(e_{\theta} \otimes s^k t^{-l}) v_0 = \begin{cases}
		0 & \text{if } l \leq k,\\
		\displaystyle\sum_{m=1}^{l-k} c(k,-l+m) (e_{\theta} \otimes t^{-m}) v_0 & \text{if } l > k,
	\end{cases}
\]
\item
\[
	(s^k t^{-l} ds) v_0 = \begin{cases}
		0 & \text{if } l \leq k,\\
		\displaystyle\sum_{m=1}^{l-k} c(k,-l+m) (t^{-m}ds) v_0 & \text{if } l > k.
	\end{cases}
\]
\end{enumerate}
\end{lem}

\begin{proof}
We prove the assertions (i) and (ii) by induction on $l$.

For $l \leq 0$, $e_{\theta} \otimes s^k t^{-l}$ is an element of $\torn \cap \tor^+$, hence it kills $v_0$.
For $l = 1$, $e_{\theta} \otimes s^k t^{-1} = f_{0,k}$ kills $v_0$ by Lemma~\ref{lem:f}.
Then we have
\begin{equation*}
	\begin{split}
		(s^k t^{-l} ds)v_0 = \left( [f_{\theta} \otimes s, e_{\theta} \otimes s^k t^{-l}] - [f_{\theta}, e_{\theta} \otimes s^{k+1}t^{-l}] \right) v_0 =0
	\end{split}
\end{equation*}
for $l \leq 1$.
We thus have proved (i) and (ii) for $l \leq 1$.

Let $l \geq 2$.
We assume the assertions (i) and (ii) for all $l' < l$.
By Lemma~\ref{lem:induction}, we have
\begin{equation}
	\begin{split}
		(e_{\theta} \otimes s^k t^{-l}) v_0 &= T_0 T_{\theta} \left( (e_{\theta} \otimes s^k t^{-l+2}) T_{\theta}^{-1} T_0^{-1} v_0 \right) \\
		&= T_0 T_{\theta} \left( (e_{\theta} \otimes s^k t^{-l+2}) T_{\theta}^{-1} (f_0 v_0) \right) \\
		&= T_0 T_{\theta} \left( (e_{\theta} \otimes s^k t^{-l+2}) T_{\theta}^{-1} (f_0) v_0 \right) \\
		&= T_0 T_{\theta} \left( T_{\theta}^{-1}(f_0)(e_{\theta} \otimes s^k t^{-l+2}) v_0 + [e_{\theta} \otimes s^k t^{-l+2}, T_{\theta}^{-1} (f_0)] v_0 \right). \label{eq:induction}
	\end{split}
\end{equation}
We have
\begin{equation*}
	\begin{split}
		[e_{\theta} \otimes s^k t^{-l+2}, T_{\theta}^{-1} (f_0)] &= [e_{\theta} \otimes s^k t^{-l+2}, -f_{\theta} \otimes t^{-1}] \\
		&=- \left( [e_{\theta} \otimes s^k t^{-l+1}, f_{\theta}] + c(k,-l+1) \right) \\
		&= [f_{\theta}, e_{\theta} \otimes s^k t^{-l+1}] - c(k,-l+1).
	\end{split}
\end{equation*}
Put
\[
	A= T_{\theta}^{-1}(f_0)(e_{\theta} \otimes s^k t^{-l+2}) v_0, \quad B= f_{\theta}(e_{\theta} \otimes s^k t^{-l+1}) v_0. 
\]
Then (\ref{eq:induction}) is equal to $T_0 T_{\theta}(A+B-c(k,-l+1)v_0)$.
By the induction assumption, we have
\[
	A= T_{\theta}^{-1}(f_0) \sum_{m=1}^{l-2-k} c(k,-l+2+m)  (e_{\theta} \otimes t^{-m}) v_0,
\]
\begin{equation*}
	\begin{split}
		B= f_{\theta} \sum_{m=1}^{l-1-k} c(k,-l+1+m)  (e_{\theta} \otimes t^{-m}) v_0 = f_{\theta} \sum_{m=0}^{l-2-k} c(k,-l+2+m)  (e_{\theta} \otimes t^{-m-1}) v_0.
	\end{split}
\end{equation*}
Then (\ref{eq:induction}) is equal to
\begin{multline}
		T_0 T_{\theta} \Bigg( \sum_{m=1}^{l-2-k} c(k,-l+2+m) \Big( T_{\theta}^{-1}(f_0) (e_{\theta} \otimes t^{-m}) + f_{\theta} (e_{\theta} \otimes t^{-m-1}) \Big) v_0 \\
		+ c(k,-l+2) f_{\theta} (e_{\theta} \otimes t^{-1}) v_0 - c(k,-l+1) v_0 \Bigg) \label{eq:induction2}
\end{multline}
if $l \geq k+2$ and to $T_0 T_{\theta}(- c(k,-l+1) v_0)$ if $l \leq k+1$.

We prove (i) for $l$.
First consider the case $l \leq k$.
In this case, we have
\begin{equation*}
	(e_{\theta} \otimes s^k t^{-l}) v_0 = T_0 T_{\theta}(- c(k,-l+1) v_0) = \dfrac{k}{-l+1} T_0 T_{\theta}( (s^{k-1} t^{-(l-1)} ds) v_0) = 0
\end{equation*}
by the induction assumption.
Hence (i) holds for $l$.
Next consider the case $l = k+1$.
In this case, we have
\begin{equation*}
	(e_{\theta} \otimes s^k t^{-l}) v_0 = T_0 T_{\theta}(- c(k,-l+1) v_0) = - c(k,-l+1) T_0 T_{\theta}(v_0).
\end{equation*}
Since we have $T_0T_{\theta} (v_0)=-f_0 v = -(e_{\theta} \otimes t^{-1})v_0$, (i) holds for $l=k+1$.
Finally consider the case $l \geq k+2$.
The equality (\ref{eq:induction}) is valid even for $k=0$ and hence we have
\[
	(e_{\theta} \otimes t^{-m-2}) v_0 = T_0 T_{\theta} \Bigg( \Big( T_{\theta}^{-1} (f_0) (e_{\theta} \otimes t^{-m}) + f_{\theta} (e_{\theta} \otimes t^{-m-1}) \Big) v_0 \Bigg)
\]
for each $m$.
This implies that (\ref{eq:induction2}) is equal to
\begin{multline*}
		\sum_{m=1}^{l-2-k} c(k,-l+2+m) (e_{\theta} \otimes t^{-m-2}) v_0\\
		+ c(k,-l+2) T_0 T_{\theta} ( f_{\theta} (e_{\theta} \otimes t^{-1}) v_0) + c(k,-l+1) (e_{\theta} \otimes t^{-1}) v_0.
\end{multline*}
Since we can easily show $T_0 T_{\theta} ( f_{\theta} (e_{\theta} \otimes t^{-1}) v_0) = (e_{\theta} \otimes t^{-2})v_0$, (i) is proved for $l$.

We prove (ii) for $l$.
By (i), we have
\begin{equation*}
	\begin{split}
		&(s^k t^{-l} ds)v_0 = \left( [f_{\theta} \otimes s, e_{\theta} \otimes s^k t^{-l}] - [f_{\theta}, e_{\theta} \otimes s^{k+1}t^{-l}] \right) v_0\\
		&= (f_{\theta} \otimes s) \sum_{m=1}^{l-k} c(k,-l+m) (e_{\theta} \otimes t^{-m}) v_0 - f_{\theta} \sum_{n=1}^{l-(k+1)} c(k+1,-l+n) (e_{\theta} \otimes t^{-n}) v_0 
	\end{split}
\end{equation*}
if $l > k$ and $(s^k t^{-l} ds)v_0 = 0$ otherwise.
Therefore we may assume $l > k$.  
We have
\begin{equation*}
	\begin{split}
		(f_{\theta} \otimes s) (e_{\theta} \otimes t^{-m}) v_0 &= [f_{\theta} \otimes s,e_{\theta} \otimes t^{-m}]v_0 \\
		&= \left( [f_{\theta}, e_{\theta} \otimes s t^{-m}] + t^{-m}ds \right) v_0 \\
		&= f_{\theta} (e_{\theta} \otimes s t^{-m}) v_0 + (t^{-m}ds) v_0 \\
		&= f_{\theta} \sum_{n=1}^{m-1} c(1,-m+n)(e_{\theta} \otimes t^{-n}) v_0 + (t^{-m}ds) v_0.
	\end{split}
\end{equation*}
We claim that
\[
	\sum_{m=1}^{l-k} c(k,-l+m) \sum_{n=1}^{m-1} c(1,-m+n)(e_{\theta} \otimes t^{-n}) v_0 = \sum_{n=1}^{l-(k+1)} c(k+1,-l+n)(e_{\theta} \otimes t^{-n}) v_0
\]
holds.
Indeed this equality is obtained by applying $h_{\theta} \otimes s$ to both sides of (i).
Hence we conclude
\begin{equation*}
	\begin{split}
		(s^k t^{-l}ds) v_0 &= \sum_{m=1}^{l-k} c(k,-l+m) \Bigg( f_{\theta} \sum_{n=1}^{m-1} c(1,-m+n)(e_{\theta} \otimes t^{-n}) v_0 + (t^{-m}ds) v_0 \Bigg)\\
		&\qquad - f_{\theta} \sum_{n=1}^{l-(k+1)} c(k+1,-l+n) (e_{\theta} \otimes t^{-n}) v_0 \\
		&= \sum_{m=1}^{l-k} c(k,-l+m) (t^{-m}ds) v_0.
	\end{split}
\end{equation*}
\end{proof}

We define the subalgebra $\bar{C}$ of $U(\tor^+)$ to be generated by $c(k,-l)$ ($k \geq 1$, $l \geq 1$).
Let $\bar{C}_1$ be the subalgebra of $\bar{C}$ generated by $c(1,-l)$ ($l \geq 1$).

\begin{lem}\label{lem:degree_one}
We have $\bar{C} v_0 = \bar{C}_1 v_0$. 
\end{lem}

\begin{proof}
Suppose $k \geq 1$ and $l \geq 1$.
We rewrite Lemma~\ref{lem:key} (ii) as
\[
	(s^{k} t^{-l} ds) v_0 = \begin{cases}
		0 & \text{if } l \leq k,\\
		\displaystyle\sum_{m=1}^{l-k} \dfrac{k}{l-m} (s^{k-1} t^{-l+m} ds) (t^{-m}ds) v_0 & \text{if } l > k.
	\end{cases}
\]
This implies that the action of $c(k+1,-l) = ((k+1)/l) s^{k}t^{-l} ds$ on $v_0$ is written in terms of a polynomial in $c(1,-m) = (1/m)t^{-m} ds$ with $m \geq 1$.
\end{proof}

\begin{lem}\label{lem:key2}
We have
\[
	\left(\affnbar^{(t)} \otimes s\bbC[s]\right) v_0 \subset \bar{C}_1 U(\affnbar^{(t)}) v_0.
\]
\end{lem}

\begin{proof}
Note that we have
\begin{equation*}
	\affnbar^{(t)} \otimes s^k = \bigoplus_{\substack{\alpha \in \Delta^+ \cup \{0\}\\ l \geq 1}} \frg_{\alpha} \otimes s^k t^{-l} \oplus  \bigoplus_{\substack{\alpha \in \Delta^- \\ l \geq 0}} \frg_{\alpha} \otimes s^k t^{-l}.
\end{equation*}
Suppose $k \geq 1$.
We show
\begin{equation}
	(x \otimes s^k t^{-l}) v_0 \in \bar{C}_1 U(\affnbar^{(t)}) v_0 \label{eq:contain}
\end{equation}
for
\begin{itemize}
\item
$x \in \frg_{\alpha}$ ($\alpha \in \Delta^+ \cup \{0\}$) and $l \geq 1$;

\item
$x \in \frg_{\alpha}$ ($\alpha \in \Delta^-$) and $l \geq 0$.
\end{itemize}
Lemma~\ref{lem:key} (i) and \ref{lem:degree_one} imply (\ref{eq:contain}) for $x=e_{\theta}$ and $l \geq 1$.
Then we obtain (\ref{eq:contain}) for $x \in \frg_{\alpha}$ ($\alpha \in \Delta^+$) and $l \geq 1$ by successively applying $f_i$'s ($i \in I$) to $(e_{\theta} \otimes s^k t^{-l}) v_0$.
We obtain (\ref{eq:contain}) for $x = h_i$ ($i \in I$) and $l \geq 1$ by applying $f_i$ to $(e_{i} \otimes s^k t^{-l}) v_0$.
We show (\ref{eq:contain}) for $x \in \frg_{\alpha}$ ($\alpha \in \Delta^-$) and $l \geq 0$.
The case $l=0$ is immediate from Lemma~\ref{lem:f}.
Assume $l \geq 1$.
We use $[h_{\alpha} \otimes s^k t^{-l}, x] = 2 x \otimes s^k t^{-l}$ and $x v_0 = 0$ to deduce
\[
	(x \otimes s^k t^{-l}) v_0 = -\dfrac{1}{2} x(h_{\alpha} \otimes s^k t^{-l}) v_0 \in x \bar{C}_1 U(\affnbar^{(t)}) v_0 \subset \bar{C}_1 U(\affnbar^{(t)}) v_0.
\]
\end{proof}

\begin{prop}\label{prop:upper_bound}
We have
\[
	W(\Lambda_0) = \bar{C}_1 U(\affnbar^{(t)}) v_0.
\]
In particular, we have an inequality
\[
	\ch_{p,q} W(\Lambda_0) \leq \ch_p L(\Lambda_0) \displaystyle\prod_{n > 0} \dfrac{1}{1-p^n q}.
\]
\end{prop}

\begin{proof}
Let $N$ be the $\bbC$-span of monomials in $\affnbar^{(t)} \otimes s\bbC[s]$.
Then the PBW theorem and Lemma~\ref{lem:degree_one} imply
\[
	W(\Lambda_0) = U(\tornbar \cap \tor^+)v_0 = \bar{C}_1 U(\affnbar^{(t)}) N v_0.
\]
Since $\affnbar^{(t)} \otimes s\bbC[s]$ is $\ad \affnbar^{(t)}$-invariant modulo central elements, we prove the assertion by Lemma~\ref{lem:key2} and \ref{lem:degree_one}. 
\end{proof}

\begin{rem}
We will show in Corollary~\ref{cor:character} that the equality
\[
	\ch_{p,q} W(\Lambda_0) = \ch_p L(\Lambda_0) \displaystyle\prod_{n > 0} \dfrac{1}{1-p^n q}
\]
holds.
\end{rem}

\begin{rem}
By Proposition~\ref{prop:character}, \ref{prop:independent} and \ref{prop:upper_bound}, we have an inequality
\[
	\ch_{p} \loc(\Lambda_0,a) \leq \ch_p L(\Lambda_0) \displaystyle\prod_{n > 0} \dfrac{1}{1-p^n}.
\]
We will show in Corollary~\ref{cor:character} that the equality holds.
In fact, we can directly prove this inequality for $\ch_{p} \loc(\Lambda_0,a)$ by a similar calculation for $\loc(\Lambda_0,a)$ instead of $W(\Lambda_0)$.
More precisely, we can show $\loc(\Lambda_0,a) = \bar{C}_1 U(\affnbar^{(t)}) v_{\Lambda_0,a}$.
Moreover, we can show that
\[
	\loc(\Lambda_0,a) = \bar{C}_0 U(\affnbar^{(t)}) v_{\Lambda_0,a}
\]
also holds, where $\bar{C}_0$ is the subalgebra of $U(\tor')$ generated by $c(0,-l)$ ($l \geq 1$).

Here we gave the calculation for $W(\Lambda_0)$ by two reasons:
\begin{enumerate}
\item
we are interested in the $(p,q)$-characters of the graded local Weyl modules for $\tor^+$;

\item
the calculation for $W(\Lambda_0)$ is easier than that for $\loc(\Lambda_0,a)$.
\end{enumerate}
\end{rem}

\section{Vertex operator construction and Weyl modules}\label{section:Vertex operator construction}

\subsection{Heisenberg Lie algebras}\label{subsection:Heisenberg}

We assume that $\frg$ is of type ADE in Section~\ref{subsection:Heisenberg} and \ref{subsection:vertex}.
Recall that $\affQ = \bigoplus_{i \in \affI} \bbZ \alpha_i$ is the root lattice of $\aff^{(t)}$.
We fix a bimultiplicative 2-cocycle $\ve \colon \affQ \times \affQ \to \{\pm 1\}$ satisfying
\[
	\ve(\alpha,\alpha) = (-1)^{(\alpha,\alpha)/2}, \quad \ve(\alpha,\beta)\ve(\beta,\alpha) = (-1)^{(\alpha,\beta)}, \quad \ve(\alpha,\delta)=1
\]
as in \cite[Section~4]{MR1066569}.
Let $\bbC[\affQ]$ be the group algebra of $\affQ$ with a $\bbC$-basis denoted by $e^{\alpha}$ ($\alpha \in \affQ$).
We make $\bbC[\affQ]$ into a $\bbC[\affQ]$-module via $\ve$, that is, we define $e^{\alpha} \cdot e^{\beta} = \ve(\alpha,\beta)e^{\alpha+\beta}$. 
We denote by $\bbC_{\ve}[\affQ]$ this module.
We define an action of $h \in \affh^{(t)}$ on $\bbC_{\ve}[\affQ]$ by $h \cdot e^{\alpha} = \langle h, \alpha \rangle e^{\alpha}$.

The toroidal Lie algebra $\tor$ contains a Heisenberg Lie algebra 
\[
	\calH = \displaystyle\bigoplus_{\substack{i \in \affI\\k \neq 0}} \bbC h_{i,k} \oplus \bbC c_s.
\]
Define the Fock representation $\affF$ of $\calH$ by
\[
	\affF = U(\calH) / \sum_{\substack{i \in \affI\\ k >0}}U(\calH) h_{i,k} + U(\calH)(c_s-1).
\]
We set
\[
	\bbV(0) = \affF \otimes \bbC_{\ve}[\affQ].
\]
Define the degree on $\bbV(0)$ by $\deg h_{i,k}=k$ and $\deg e^{\alpha}= (\alpha,\alpha)/2$.
Then we regard $\bbV(0)$ as a module of $\torh = \calH \oplus \affh^{(t)} \oplus \bbC d_s$ via the actions of $\calH$ and $\affh^{(t)}$ on $\affF$ and $\bbC_{\ve}[\affQ]$ respectively, and so that $d_s$ counts the degree.

Similarly we define $\mathcal{F}$ to be the Fock representation for a Heisenberg Lie subalgebra
\[
	\displaystyle\bigoplus_{\substack{i \in I\\k \neq 0}} \bbC h_{i,k} \oplus \bbC c_s
\]
of $\aff^{(s)}$.

\subsection{Vertex representations}\label{subsection:vertex}

For each $\alpha \in \affDelta$, we set
\[
	X(\alpha,u) = u^{(\alpha,\alpha)/2} \left( e^{\alpha} u^{h_{\alpha}} \right) \exp\left( \sum_{k>0} \dfrac{h_{\alpha} \otimes s^{-k}}{k} u^{k} \right) \exp\left( -\sum_{k>0} \dfrac{h_{\alpha} \otimes s^{k}}{k} u^{-k} \right)
\]
as an element of $( \End_{\bbC} \bbV(0) )[[u^{\pm1}]]$.
Here $u^{h_{\alpha}}$ acts by
\[
	u^{h_{\alpha}} \cdot e^{\beta} = u^{(\alpha,\beta)} e^{\beta}.
\]
Define $X_{k}(\alpha)$ by the expansion
\[
	X(\alpha,u) = \sum_{k \in \bbZ} X_k(\alpha) u^{-k}.
\]

\begin{thm}[\cite{MR1066569} Proposition~4.3]\label{thm:MEY}
We can extend the action of $\torh = \calH \oplus \affh^{(t)} \oplus \bbC d_s$ to $\tor$ on $\bbV(0)$ by
\[
	e_{i,k} \mapsto X_{k}(\alpha_i), \quad f_{i,k} \mapsto X_{k}(-\alpha_i).
\]
\end{thm}

We denote by $\tau$ the action of $c(0,1)$ on $\bbV(0)$.
Then by \cite[(4.1) and Proposition~5.3 (ii)]{MR1066569}, the action of $c(0,k)$ for $k \neq 0$ is given by $\tau^k$.
The subalgebra of $\End_{\bbC} \bbV(0)$ generated by $\tau^k$ ($k \in \bbZ$) is isomorphic to the Laurent polynomial algebra $\bbC[\tau^{\pm 1}]$. 

We denote by $\delta(k)$ the action of $c(k,0)$ on $\bbV(0)$ for $k<0$.
They freely generate a polynomial subalgebra of $\End_{\bbC} \bbV(0)$ and we denote it by $D$. 
We have an isomorphism of $\bbC$-vector spaces 
\[
	\affF \cong \mathcal{F} \otimes D.
\]

\begin{prop}[\cite{MR1066569} Lemma~5.6]\label{prop:freeness_vertex_rep}
The multiplication map gives an isomorphism
\[
	\bbV(0) \cong \mathcal{F} \otimes \bbC_{\ve}[Q] \otimes D \otimes \bbC[\tau^{\pm 1}]
\]
of $\bbC$-vector spaces.
In particular, $\bbV(0)$ is free over $\bbC[\tau^{\pm 1}]$.
\end{prop}

The $\aff^{(s)}$-submodule $\mathcal{F} \otimes \bbC_{\ve}[Q]$ is known to be isomorphic to the level one integrable irreducible $\aff^{(s)}$-module $L(\Lambda_0)^{(s)}$ with highest weight $\Lambda_0$ by Frenkel-Kac \cite{MR595581}. 
Hence it has the following defining relations:
\begin{gather}
	(f_{\theta} \otimes s) (1 \otimes e^0) = 0,\quad e_i (1 \otimes e^0) = 0 \ (i \in I), \label{eq:Frenkel-Kac1}\\
	c_s (1 \otimes e^0) = 1 \otimes e^0,\quad h_i (1 \otimes e^0) = 0 \ (i \in I),\quad d_s (1 \otimes e^0) = 0,\label{eq:Frenkel-Kac2}\\
	(e_{\theta} \otimes s^{-1})^2 (1 \otimes e^0) = 0,\quad f_i (1 \otimes e^0) = 0 \ (i \in I).\label{eq:Frenkel-Kac3}
\end{gather}
We will determine the defining relations of $\bbV(0)$ as a $\tor$-module as a main result of this article.

\subsection{General construction}

We review the construction of $\tor$-modules given by Iohara-Saito-Wakimoto~\cite{MR1688100} and Eswara Rao~\cite{MR3076215}.
Assume that $\frg$ is an arbitrary simple Lie algebra.
Let $D$ be the polynomial algebra generated by the elements $\delta(k)$ ($k < 0$). 
For a given smooth $\aff^{(s)}$-module $M$, we will define a $\tor$-module structure on
\[
	M \otimes D \otimes \bbC[\tau^{\pm 1}]
\]
as follows.
For an element $x$ of $\frg$, we put $x(u) = \sum_{k \in \bbZ} (x \otimes s^k) u^{-k}$.
Define a formal series $\Delta_l(u)$ for each $l \in \bbZ$ by
\[
	\Delta_l(u) = \exp \left( \sum_{k > 0} \dfrac{l \delta(-k)}{k} u^{k} \right).
\]
We make $D$ into a graded algebra by $\deg \delta(k) = k$ and let $d^{(D)}$ be the operator which counts the degree on $D$.
We make $\bbC[\tau^{\pm 1}]$ into a graded algebra by $\deg \tau = 1$ and let $d^{(\tau)}$ be the operator which counts the degree on $\bbC[\tau^{\pm 1}]$.

\begin{thm}[\cite{MR1688100} Lemma~2.1, \cite{MR3076215} Theorem~4.1]\label{thm:ISW-E}
Let $M$ be a smooth $\aff^{(s)}$-module.
The assignment
\[
	\sum_{k \in \bbZ} (x \otimes s^k t^l) u^{-k} \mapsto x(u) \otimes \Delta_l(u) \otimes \tau^l
\]
for $x \in \frg,$
\[
	\sum_{k \in \bbZ} (s^{k-1} t^l ds) u^{-k} \mapsto c_s \otimes \Delta_l(u) \otimes \tau^l, \quad 
	s^{k} t^{-1} dt \mapsto \begin{cases}
		\id \otimes \delta(k) \otimes \id & \text{ if } k < 0,\\
		0 & \text{ if } k \geq 0,
	\end{cases}
\]
\[
	d_s \mapsto d_s \otimes \id \otimes \id + \id \otimes d^{(D)} \otimes \id, \quad d_t \mapsto \id \otimes \id \otimes d^{(\tau)}
\]
gives a $\tor$-module structure on $M \otimes D \otimes \bbC[\tau^{\pm 1}]$.
\end{thm}

\begin{rem}
Let us give a remark on the results of \cite{MR1688100} and \cite{MR3076215} stated above.
In \cite{MR1688100}, the authors consider a Lie algebra bigger than $\tor$ and the module they construct is bigger than $M \otimes D \otimes \bbC[\tau^{\pm 1}]$.
If one restricts the action to $\tor$, we can take $M \otimes D \otimes \bbC[\tau^{\pm 1}]$ as a $\tor$-submodule.
Moreover, although they assume that $\frg$ is of type ADE in \cite{MR1688100}, the construction does not need the assumption.
Later this construction of $\tor$-modules has been generalized in \cite{MR3076215} to some Lie superalgebras.   
\end{rem}

Take $M$ as the level one integrable irreducible $\aff^{(s)}$-module  $L(\Lambda_0)^{(s)}$ with highest weight $\Lambda_0$ and set
\[
	\bbV(0) = L(\Lambda_0)^{(s)} \otimes D \otimes \bbC[\tau^{\pm 1}].
\]
This definition is compatible with the construction given in Section~\ref{subsection:Heisenberg} and \ref{subsection:vertex} if $\frg$ is of type ADE.
Indeed, the definition of the vertex operator $X(\alpha,u)$ implies that
\[
	X(\beta+l\delta,u) = \begin{cases}
		X(\beta,u) \otimes \Delta_l(u) \otimes \tau^l & \text{if } \beta \in \Delta,\\
		\id \otimes \Delta_l(u) \otimes \tau^l & \text{if } \beta = 0,
	\end{cases}
\]
when we write $\alpha \in \affDelta$ as $\alpha = \beta + l\delta$ with $\beta \in \Delta \cup \{0\}$ and $l \in \bbZ$.

Let $v^{(s)}$ be a highest weight vector of $L(\Lambda_0)^{(s)}$. 
We generalize the relations given in (\ref{eq:Frenkel-Kac1}), (\ref{eq:Frenkel-Kac2}), (\ref{eq:Frenkel-Kac3}).

\begin{lem}\label{lem:highest}
We have
\begin{gather}
	(f_{\theta} \otimes s) (v^{(s)} \otimes 1 \otimes 1) = 0,\quad e_i (v^{(s)} \otimes 1 \otimes 1) = 0 \ (i \in I), \label{eq:Frenkel-Kac1new}\\
	c_s (v^{(s)} \otimes 1 \otimes 1) = v^{(s)} \otimes 1 \otimes 1, \quad h_i (v^{(s)} \otimes 1 \otimes 1) = 0 \ (i \in I),\quad d_s (v^{(s)} \otimes 1 \otimes 1) = 0,\label{eq:Frenkel-Kac2new}\\
	(e_{\theta} \otimes s^{-1})^2 (v^{(s)} \otimes 1 \otimes 1) = 0,\quad f_i (v^{(s)} \otimes 1 \otimes 1) = 0 \ (i \in I).\label{eq:Frenkel-Kac3new}
\end{gather}
\end{lem}

\begin{proof}
These are direct consequences of the definition of the action and the relations in $L(\Lambda_0)^{(s)}$.
\end{proof}

\begin{lem}\label{lem:vertex}
We have $\aff^{(t)} (v^{(s)} \otimes 1 \otimes 1) = 0$.
\end{lem}

\begin{proof}
We have $\frg (v^{(s)} \otimes 1 \otimes 1) = (\frg v^{(s)}) \otimes 1 \otimes 1 = 0$.
To see the action of $e_0 = f_{\theta} \otimes t$, consider the assignment
\[
	\sum_{k \in \bbZ} (f_{\theta} \otimes s^k t) u^{-k} \mapsto f_{\theta} (u) \otimes \Delta_1(u) \otimes \tau.
\]
Expand $\Delta_1(u) = \sum_{k \geq 0} \Delta_1^{(-k)} u^k$.
Then the action of $e_0 = f_{\theta} \otimes t$ is given by $\sum_{k \geq 0} (f_{\theta}\otimes s^k) \otimes \Delta_1^{(-k)} \otimes \tau$.
Since we have $(f_{\theta}\otimes s^k) v^{(s)} = 0$ for $k \geq 0$, we have $e_0(v^{(s)} \otimes 1 \otimes 1)=0$.
Similarly the action of $f_0 = e_{\theta} \otimes t^{-1}$ is given by $\sum_{k \geq 0} (e_{\theta}\otimes s^k) \otimes \Delta_{-1}^{(-k)} \otimes \tau^{-1}$, hence it acts on $v^{(s)} \otimes 1 \otimes 1$ by $0$.
We have $c_t (v^{(s)} \otimes 1 \otimes 1) = 0$ and $d_t (v^{(s)} \otimes 1 \otimes 1) = 0$ by the definition of the action of $c_t$ and $d_t$.
\end{proof}

\subsection{Isomorphisms}

We define a $\tor$-module $\bbV$ by the pull-back of $\bbV(0)$ via the automorphism $S^{-1}$, that is, $\bbV = (S^{-1})^*\bbV(0)$.
Denote the vector of $\bbV$ corresponding to $v^{(s)} \otimes 1 \otimes 1 \in \bbV(0)$ by $\bfv$.

The action of $c(1,0)$ on $\bbV$ corresponds to $\tau^{-1}$ on $\bbV(0)$ via $S^{-1}$ since $S^{-1}(c(1,0)) = c(0,-1)$.
We regard $\bbV$ as a module over $A(\Lambda_0)=\bbC[z^{\pm 1}]$ via $z \mapsto c(1,0)$ and then $\bbV$ becomes a free $A(\Lambda_0)$-module by Proposition~\ref{prop:freeness_vertex_rep}.
We put $\bbV_a = \bbV \otimes_{A(\Lambda_0)} \bbC_a$ for $a \in \bbC^{\times}$.
This $\bbV_a$ is a $\tor'$-module.
The character of $\bbV_a$ is given as follows.

\begin{prop}\label{prop:character_V}
We have $\ch_p \bbV_a = \ch_p L(\Lambda_0) \displaystyle\prod_{n > 0} \dfrac{1}{1-p^n}$.
\end{prop}

\begin{proof}
The assertion obviously follows from the construction of the action of $\tor$ on $\bbV(0) = L(\Lambda_0)^{(s)} \otimes D \otimes \bbC[\tau^{\pm 1}]$. 
\end{proof}

Let us study relation between the level one global Weyl module $\glob(\Lambda_0)$ and $\bbV$. 

\begin{lem}\label{lem:relation}
We have 
\[
	h_{i,k} \bfv = \begin{cases} 0 & \text{if } i \in I, \\ z^k \bfv & \text{if } i=0 \end{cases}
\]
for any $k \in \bbZ$.	
\end{lem}

\begin{proof}
We have 
\[
	S^{-1}(h_{i,k}) = \begin{cases} h_i \otimes t^{-k} & \text{if } i \in I, \\ s^{-1} t^{-k} ds - h_{\theta} \otimes t^{-k} & \text{if } i=0. \end{cases}
\]
By Lemma~\ref{lem:vertex}, we have $(h_i \otimes t^{-k}) (v^{(s)} \otimes 1 \otimes 1) = (h_{\theta} \otimes t^{-k}) (v^{(s)} \otimes 1 \otimes 1) =0$.
Since we have $(s^{-1} t^{-k} ds) (v^{(s)} \otimes 1 \otimes 1) = \tau^{-k} (v^{(s)} \otimes 1 \otimes 1)$ and $\tau^{-1}$ corresponds to $z$, the assertion is proved.
\end{proof}

\begin{lem}\label{lem:surjection}
We have a surjective homomorphism $\glob(\Lambda_0) \to \bbV$ of modules over both $\tor$ and $A(\Lambda_0)$.
\end{lem}

\begin{proof}
The equalities (\ref{eq:Frenkel-Kac1new}), (\ref{eq:Frenkel-Kac2new}), (\ref{eq:Frenkel-Kac3new}) are equivalent to
 \begin{gather*}
	e_i \bfv = 0 \ (i \in \affI), \\
	c_t \bfv = \bfv, \quad h_i \bfv = 0 \ (i \in I),\quad d_t \bfv = 0,\\
	f_0^2 \bfv = 0,\quad f_i \bfv = 0 \ (i \in I).
\end{gather*}
Moreover we have 
\begin{align*}
	c_s \bfv &= S^{-1}(c_s)(v^{(s)} \otimes 1 \otimes 1) = -c_t (v^{(s)} \otimes 1 \otimes 1) = 0,\\
	d_s \bfv &= S^{-1}(d_s)(v^{(s)} \otimes 1 \otimes 1) = -d_t (v^{(s)} \otimes 1 \otimes 1) = 0
\end{align*}
by Lemma~\ref{lem:vertex}.
We need to check $e_{i,k} \bfv = 0$ for $i \in \affI$ and $k \in \bbZ$.
This follows from $e_i \bfv = 0$ and Lemma~\ref{lem:relation}.
\end{proof}

By Lemma~\ref{lem:surjection}, we have a surjective $\tor'$-homomorphism $\loc(\Lambda_0,a) \to \bbV_a$ for every $a \in \bbC^{\times}$. 
Hence we have inequalities of the characters
\begin{equation}
	\ch_p \loc^+(\Lambda_0,a) \geq \ch_p \loc(\Lambda_0,a) \geq \ch_p \bbV_a \label{eq:inequality}
\end{equation}
by Proposition~\ref{prop:character}.

\begin{thm}\label{thm:main}
We have isomorphisms
\[
	\glob(\Lambda_0) \stackrel{\cong}{\longrightarrow} \bbV, \qquad \loc(\Lambda_0,a) \stackrel{\cong}{\longrightarrow} \bbV_a
\]
of modules over $\tor$ and $\tor'$ respectively.
\end{thm}

\begin{proof}
First we prove the isomorphism $\loc(\Lambda_0,a) \cong \bbV_a$.
We have
\begin{equation}
	\ch_p \loc^+(\Lambda_0,a) = \ch_p W(\Lambda_0) \leq \ch_p L(\Lambda_0) \prod_{n>0} \dfrac{1}{1-p^n} = \ch_p \bbV_a \label{eq:inequality2}
\end{equation}
by Proposition~\ref{prop:independent}, \ref{prop:upper_bound}, \ref{prop:character_V}.
Then the inequalities (\ref{eq:inequality}) and (\ref{eq:inequality2}) imply $\ch_p \loc(\Lambda_0,a) = \ch_p \bbV_a$.
This shows that the surjective homomorphism $\loc(\Lambda_0,a) \to \bbV_a$ is an isomorphism for every $a \in \bbC^{\times}$.
Next we prove the isomorphism $\glob(\Lambda_0) \cong \bbV$.
Since $\bbV$ is a free $A(\Lambda_0)$-module, we can take a splitting of the exact sequence
\[
	0 \to \Ker \to \glob(\Lambda_0) \to \bbV \to 0
\]
of $A(\Lambda_0)$-modules.
The isomorphism $\loc(\Lambda_0,a) \cong \bbV_a$ implies $\Ker \otimes_{A(\Lambda_0)} \bbC_a = 0$ for every $a \in \bbC^{\times}$.
Then by Nakayama's lemma, we see that $\Ker = 0$ and obtain the isomorphism $\glob(\Lambda_0) \cong \bbV$.
\end{proof}

\begin{cor}\label{cor:character}
We have
\[
	\ch_{p} \loc(\Lambda_0,a) = \ch_{p} \loc^+(\Lambda_0,a) = \ch_p L(\Lambda_0) \left( \prod_{n>0} \dfrac{1}{1-p^n} \right)
\]
for $a \in \bbC^{\times}$ and
\[
	\ch_{p,q} W(\Lambda_0) = \ch_p L(\Lambda_0) \left( \prod_{n>0} \dfrac{1}{1-p^n q} \right).
\]
\end{cor}

\begin{proof}
The equalities for the $p$-characters are verified in the proof of Theorem~\ref{thm:main}.
The equality for the $(p,q)$-character follows from that for the $p$-character and Proposition~\ref{prop:upper_bound}. 
\end{proof}

\newcommand{\etalchar}[1]{$^{#1}$}
\def\cprime{$'$} \def\cprime{$'$} \def\cprime{$'$} \def\cprime{$'$}
\providecommand{\bysame}{\leavevmode\hbox to3em{\hrulefill}\thinspace}
\providecommand{\MR}{\relax\ifhmode\unskip\space\fi MR }
\providecommand{\MRhref}[2]{%
  \href{http://www.ams.org/mathscinet-getitem?mr=#1}{#2}
}
\providecommand{\href}[2]{#2}


\begin{thebibliography}{LNSSS}

\bibitem[CFK]{MR2718936}
Vyjayanthi Chari, Ghislain Fourier, and Tanusree Khandai, \emph{A categorical
  approach to {W}eyl modules}, Transform. Groups \textbf{15} (2010), no.~3,
  517--549.

\bibitem[CIK]{MR3384485}
Vyjayanthi Chari, Bogdan Ion, and Deniz Kus, \emph{Weyl modules for the
  hyperspecial current algebra}, Int. Math. Res. Not. IMRN (2015), no.~15,
  6470--6515. %\MR{3384485}

\bibitem[CLe]{MR2017585}
Vyjayanthi Chari and Thang Le, \emph{Representations of double affine {L}ie
  algebras}, A tribute to {C}. {S}. {S}eshadri ({C}hennai, 2002), Trends Math.,
  Birkh\"{a}user, Basel, 2003, pp.~199--219. %\MR{2017585}

\bibitem[CLo]{MR2271991}
Vyjayanthi Chari and Sergei Loktev, \emph{Weyl, {D}emazure and fusion modules
  for the current algebra of {$\mathfrak{sl}_{r+1}$}}, Adv. Math. \textbf{207}
  (2006), no.~2, 928--960.

\bibitem[CP]{MR1850556}
Vyjayanthi Chari and Andrew Pressley, \emph{Weyl modules for classical and
  quantum affine algebras}, Represent. Theory \textbf{5} (2001), 191--223
  (electronic).

\bibitem[E]{MR3076215}
Senapathi Eswara~Rao, \emph{A new class of modules for toroidal {L}ie superalgebras},
  S\~{a}o Paulo J. Math. Sci. \textbf{6} (2012), no.~1, 97--115. %\MR{3076215}

\bibitem[FJMM]{MR3023228}
Boris Feigin, Michio Jimbo, Tetsuji Miwa, and Evgeny Mukhin,
  \emph{Representations of quantum toroidal {${\mathfrak{gl}}_n$}}, J. Algebra
  \textbf{380} (2013), 78--108. %\MR{3023228}

\bibitem[FK]{MR595581}
Igor~B. Frenkel and Victor~G. Kac, \emph{Basic representations of affine {L}ie
  algebras and dual resonance models}, Invent. Math. \textbf{62} (1980/81),
  no.~1, 23--66. %\MR{595581}

\bibitem[FeLo]{MR2102326}
Boris Feigin and Sergei Loktev, \emph{Multi-dimensional {W}eyl modules and
  symmetric functions}, Comm. Math. Phys. \textbf{251} (2004), no.~3, 427--445.

\bibitem[FoLi]{MR2323538}
Ghislain Fourier and Peter Littelmann, \emph{Weyl modules, {D}emazure modules,
  {KR}-modules, crystals, fusion products and limit constructions}, Adv. Math.
  \textbf{211} (2007), no.~2, 566--593.

\bibitem[GRW]{MR4014633}
Nicolas Guay, Vidas Regelskis, and Curtis Wendlandt, \emph{Vertex
  representations for {Y}angians of {K}ac-{M}oody algebras}, J. \'{E}c.
  polytech. Math. \textbf{6} (2019), 665--706. %\MR{4014633}

\bibitem[I]{MR1953294}
Bogdan Ion, \emph{Nonsymmetric {M}acdonald polynomials and {D}emazure
  characters}, Duke Math. J. \textbf{116} (2003), no.~2, 299--318. %\MR{1953294}

\bibitem[ISW]{MR1688100}
Kenji Iohara, Yoshihisa Saito, and Minoru Wakimoto, \emph{Hirota bilinear forms
  with {$2$}-toroidal symmetry}, Phys. Lett. A \textbf{254} (1999), no.~1-2,
  37--46. %\MR{1688100}

\bibitem[LNSSS]{MR3674171}
Cristian Lenart, Satoshi Naito, Daisuke Sagaki, Anne Schilling, and Mark
  Shimozono, \emph{A uniform model for {K}irillov-{R}eshetikhin crystals {II}.
  {A}lcove model, path model, and {$P=X$}}, Int. Math. Res. Not. IMRN (2017),
  no.~14, 4259--4319. %\MR{3674171}

\bibitem[M]{MR1693755}
Kei Miki, \emph{Toroidal braid group action and an automorphism of toroidal
  algebra {${\rm U}_q({\rm sl}_{n+1,tor})\ (n\geq 2)$}}, Lett. Math. Phys.
  \textbf{47} (1999), no.~4, 365--378. %\MR{1693755}

\bibitem[MEY]{MR1066569}
Robert~V. Moody, Senapathi Eswara~Rao, and Takeo Yokonuma, \emph{Toroidal {L}ie
  algebras and vertex representations}, Geom. Dedicata \textbf{35} (1990),
  no.~1-3, 283--307. %\MR{1066569}

\bibitem[N]{MR2855081}
Katsuyuki Naoi, \emph{Weyl modules, {D}emazure modules and finite crystals for
  non-simply laced type}, Adv. Math. \textbf{229} (2012), no.~2, 875--934.
  %\MR{2855081}

\bibitem[Sai]{MR1617066}
Yoshihisa Saito, \emph{Quantum toroidal algebras and their vertex
  representations}, Publ. Res. Inst. Math. Sci. \textbf{34} (1998), no.~2,
  155--177. %\MR{1617066}

\bibitem[San]{MR1771615}
Yasmine~B. Sanderson, \emph{On the connection between {M}acdonald polynomials
  and {D}emazure characters}, J. Algebraic Combin. \textbf{11} (2000), no.~3,
  269--275. %\MR{1771615}

\bibitem[STU]{MR1603798}
Yoshihisa Saito, Kouichi Takemura, and Denis Uglov, \emph{Toroidal actions on
  level {$1$} modules of {$U_q(\widehat{\mathfrak{sl}}_n)$}}, Transform. Groups
  \textbf{3} (1998), no.~1, 75--102.

\bibitem[T]{MR3908899}
Alexander Tsymbaliuk, \emph{Several realizations of {F}ock modules for toroidal
  {$\ddot {U}_{q,d}(\mathfrak{sl}_n)$}}, Algebr. Represent. Theory \textbf{22}
  (2019), no.~1, 177--209. %\MR{3908899}

\end{thebibliography}
\end{document}